\documentclass{article}

\usepackage[utf8]{inputenc}
\usepackage[T1]{fontenc}
\usepackage{lmodern}
\usepackage{amsmath} 
\usepackage{amssymb} 
\usepackage{amsfonts}
\usepackage{amsthm}  
\usepackage{xcolor}
\usepackage{mathtools}

\include{macros}
\theoremstyle{definition}
\newtheorem{definition}{Definition}
\theoremstyle{theorem}
\newtheorem{theorem}{Theorem}

\newtheorem{lemmach}{Lemma}
\theoremstyle{theorem}

\theoremstyle{remark}

\theoremstyle{remark}

\theoremstyle{example}

\theoremstyle{notation}
\newtheorem{notation}{Notation}

\title{Sections and Chapters}
\title{H-coloring Dichotomy in Proof Complexity}
\author{Azza Gaysin\footnote{This work has been supported by Charles University
   Research Centre program No.UNCE/SCI/022.}}
\date{%
    Department of Algebra, Faculty of Mathematics and Physics\\%
    Charles University in Prague\\%
}

\begin{document}
\allowdisplaybreaks

\maketitle

\begin{abstract}
The $\mathcal{H}$-coloring problem for undirected simple graphs is a computational problem from a huge class of the constraint satisfaction problems (CSP): an $\mathcal{H}$-coloring of a graph $\mathcal{G}$ is just a homomorphism from $\mathcal{G}$ to $\mathcal{H}$ and the problem is to decide for fixed $\mathcal{H}$, given $\mathcal{G}$, if a homomorphism exists or not.

The dichotomy theorem for the $\mathcal{H}$-coloring problem was proved by Hell and Nešetřil \cite{10.1016/0095-8956(90)90132-J} in 1990 (an analogous theorem for all CSPs was recently proved by Zhuk \cite{zhuk2017proof} and Bulatov \cite{bulatov2017dichotomy}) and it says that for each $\mathcal{H}$ the problem is either $p$-time decidable or $NP$-complete. Since negations of unsatisfiable instances of CSP can be expressed as propositional tautologies, it seems to be natural to investigate the proof complexity of CSP.

We show that the decision algorithm in the $p$-time case of the $\mathcal{H}$-coloring problem can be formalized in a relatively weak theory and that the tautologies expressing the negative instances for such $\mathcal{H}$ have short proofs in propositional proof system $R^*(log)$, a mild extension of resolution. In fact, when the formulas are expressed as unsatisfiable sets of clauses they have $p$-size resolution proofs. To establish this we use a well-known connection between theories of bounded arithmetic and propositional proof systems.

We complement this result by a lower bound result that holds for many weak proof systems for a special example of $NP$-complete case of the $\mathcal{H}$-coloring problem, using known results about proof complexity of the Pigeonhole Principle.

\end{abstract}

\section{Introduction}

The constraint satisfaction problem (CSP) is a computational problem. The problem is in finding an assignment of values to a set of variables, such that this assignment satisfies some specified feasibility conditions. If such assignment exists, we call the instance of CSP satisfiable and unsatisfiable otherwise. One can also define CSP through the homomorphism between relational structures: in the constraint satisfaction problem associated with a structure $\mathcal{H}$, denoted by CSP($\mathcal{H}$) the question is, given a structure $\mathcal{G}$ over the same vocabulary, whether there exists a homomorphism from $\mathcal{G}$ to $\mathcal{H}$. It turns out, that all CSPs can be classified with only two complexity classes: there are either polynomial-time CSPs, or $NP$-complete CSPs. This dichotomy was conjectured by Feder and Vardi in 1998 \cite{1} and recently proved by Zhuk \cite{zhuk2017proof} and Bulatov \cite{bulatov2017dichotomy}.

The $\mathcal{H}$-coloring problem is essentially CSP($\mathcal{H}$) on relational structures that are undirected graphs. Its computational complexity was investigated years ago and the Dichotomy theorem for the $\mathcal{H}$-coloring problem was proved by Hell and Nešetřil \cite{10.1016/0095-8956(90)90132-J} in 1990.

\begin{theorem}[The Dichotomy theorem for the $\mathcal{H}$-coloring problem, \cite{10.1016/0095-8956(90)90132-J}]\label{TWP1AC01}
If $\mathcal{H}$ is bipartite then the $\mathcal{H}$-coloring problem is in $P$. Otherwise the $\mathcal{H}$-coloring problem is $NP$-complete.
\end{theorem}
There is an easy $\mathcal{H}$-colorability test when $\mathcal{H}$ is bipartite: 
\begin{lemmach}[\cite{10.1016/0095-8956(90)90132-J}]\label{L1} For all graphs $\mathcal{G},\mathcal{H}$ if $\mathcal{H}$ is bipartite, then $\mathcal{G}$ is $\mathcal{H}$-colorable if and only if $\mathcal{G}$ is bipartite graph. 
\end{lemmach} 
Instances of CSP($\mathcal{H}$) can be expressed by propositional formulas: denote by $\alpha(\mathcal{G},\mathcal{H})$ the propositional formula expressing that there is a homomorphism from $\mathcal{G}$ to $\mathcal{H}$ (see Definition \ref{DefinitionCNF}). If the instance of CSP is unsatisfiable, then $\neg\alpha(\mathcal{G},\mathcal{H})$ is a tautology (for the $\mathcal{H}$-coloring problem we get a tautology every time we consider bipartite graph $\mathcal{H}$ and non-bipartite graph $\mathcal{G}$). From this point of view it is natural to ask about its proof complexity. Common way to do this is to formalize the sentence in some weak theory of bounded arithmetic and first prove that this universal statement is valid in all finite structures. Then it could be translated into a family of propositional tautologies, that will have short proofs in the corresponding proof system. The simpler the theory is, the weaker propositional proof system will be.

If $\mathcal{H}$-coloring is $NP$-complete then the negative instances (graphs $\mathcal{G}$ that cannot be $\mathcal{H}$-colored) form a $coNP$-complete set and hence, unless $NP = coNP$, they cannot have poly-size proofs in any
propositional proof system. In the case when $\mathcal{H}$-coloring is tractable (i.e. we have a $p$-time algorithm distinguishing positive and negative instances) we shall prove that the negative instances, when represented by unsatisfiable sets of clauses, actually have $p$-size resolution refutations. A resolution proof is a much more
rudimentary object than a run of a $p$-time algorithm: it operates just with clauses. (In fact, the algorithm can be reconstructed from the proof via feasible interpolation, Sec.3.3.2.)

In this paper we show, that the decision algorithm in the $p$-time case of the $\mathcal{H}$-coloring problem (that is, the case where $\mathcal{H}$ is a bipartite graph) can be formalized in a relatively weak two-sorted theory $V^0$ \cite{10.5555/1734064}, which is quite convenient  for formalizing sets of vertices and relations between them, and proved by using only formulas of restricted complexity in the Induction scheme. The tautologies expressing the negative instances for such $\mathcal{H}$ hence have short proofs in propositional proof system $R^*(log)$, a mild extension of resolution. In fact, when the formulas are expressed as unsatisfiable sets of clauses they have $p$-size resolution proofs.

We shall complement this upper bound by a lower bound, by giving examples of graphs $\mathcal{H}$ and $\mathcal{G}$ for which CSP($\mathcal{H}$) is $NP$-complete and for which any proof of the tautologies expressing that $\mathcal{G} \notin$ CSP($\mathcal{H}$) must have exponential size length in constant-depth Frege system (which contains $R^*(log)$) and some other well-known proof systems. This is based on the proof complexity of the Pigeonhole Principle.

The paper is organized as follows. In Section 2 we give some common definitions from propositional proof complexity and theory of bounded arithmetic, the definition of CSP in terms of homomorphisms and explain how to express instances of CSP by propositional formulas. In Section 3 we formalize the  $\mathcal{H}$-coloring problem in theory $V^0$ and prove all auxiliary lemmas and the main universal statement. Then we proceed with translation of the main universal statement into propositional tautologies and prove that for any non-bipartite graph $\mathcal{G}$ and bipartite graph $\mathcal{H}$ the propositional family, expressing that there is no homomorphism from $\mathcal{G}$ to $\mathcal{H}$, has polynomial size bounded depth Frege proofs. Some definitions and material here about translations are quite standard in proof complexity but maybe not so in the CSP community, hence we decided to include them explicitly. We end the Section with some remarks about collateral result and minor improvement of the upper bound. 
In Section 4 we consider $NP$-complete case of the $\mathcal{H}$-coloring problem and known lower bounds for one suitable example. In Section 5, we discuss open questions and further direction of research.

\section{Preliminaries}
\subsection{Constraint satisfaction problems and the $\mathcal{H}$-coloring problem}
There are many equivalent definitions of the constraint satisfaction problem. Here we will use the definition in terms of homomorphisms.

\begin{definition}[Constraint satisfaction problem] $\,$

\begin{itemize}
\item A \emph{vocabulary} is a finite set of relational symbols $R_1$,..., $R_n$ each of which has a fixed arity.
\item A \emph{relational structure} over the vocabulary $R_1$,..., $R_n$ is the tuple $\mathcal{H} = (H, R^\mathcal{H}_1,...,R^\mathcal{H}_n)$ s.t. $H$ is non-empty set , called the \emph{universe} of $\mathcal{H}$, and each $R^\mathcal{H}_i$ is a relation on $H$ having the same arity as the symbol $R_i$.
\item For $\mathcal{G}$, $\mathcal{H}$ being relational structures over the same vocabulary $R_1$,..., $R_n$ a homomorphism from  $\mathcal{G}$ to $\mathcal{H}$ is a mapping $\phi: \mathcal{G} \rightarrow \mathcal{H}$ from the universe $G$ to $H$ s.t., for every $m$-ary relation $R^\mathcal{G}$ and every tuple $(a_1,...,a_m) \in R^\mathcal{G}$ we have $(\phi(a_1),...,\phi(a_m)) \in R^\mathcal{H}$.
\end{itemize}
Let $\mathcal{H}$ be a relational structure over a vocabulary $R_1$,..., $R_n$. In the \emph{constraint satisfaction problem} associated with $\mathcal{H}$, denoted by CSP($\mathcal{H}$) the question is, given a structure $\mathcal{G}$ over the same vocabulary, whether there exists a homomorphism from $\mathcal{G}$ to $\mathcal{H}$. If the answer is positive, then we call the instance $\mathcal{G}$ \emph{satisfiable} and \emph{unsatisfiable} otherwise \cite{BULATOV200531}.
\end{definition}

The $\mathcal{H}$-coloring problem could be described as follows: let $\mathcal{H} = (V_\mathcal{H},E_\mathcal{H})$ be a simple undirected graph without loops, whose vertices we consider as different colors. An $\mathcal{H}$-coloring of a graph $\mathcal{G} = (V_\mathcal{G},E_\mathcal{G})$ is an assignment of colors to the vertices of $\mathcal{G}$ such that adjacent vertices of $\mathcal{G}$ obtain adjacent colors. Since a graph homomorphism $h: \mathcal{G} \to \mathcal{H}$ is a mapping of $V_\mathcal{G}$ to $V_\mathcal{H}$ such that if $g,g'$ are adjacent vertices of $\mathcal{G}$, then so are $h(g),h(g')$, it is easy to see that an $\mathcal{H}$-coloring of $\mathcal{G}$ is just a homomorphism $\mathcal{G} \to \mathcal{H}$. A simple undirected graph $\mathcal{H}$ can be considered as a relational structure $\mathcal{H} = (V_\mathcal{H},E_\mathcal{H})$ with only one binary symmetric relation $E_\mathcal{H}(i,j)$ (to $i,j$ be adjacent vertices). Thus, the problem of $\mathcal{H}$-coloring of a graph $\mathcal{G}$ is equivalent to CSP($\mathcal{H}$).

To express an instance of CSP($\mathcal{H}$) by propositional formula we use the following construction \cite{atserias2017proof}. For any sets $V_\mathcal{G}$ and $V_\mathcal{H}$ by $V(V_\mathcal{G}, V_\mathcal{H})$ we denote a set of propositional variables: for every $v \in V_\mathcal{G}$ and every $u\in V_\mathcal{H}$ there is a variable $x_{v,u}$ in the set $V(V_\mathcal{G}, V_\mathcal{H})$. A variable $x_{v,u}$ is assigned the truth value $1$ if and only if the vertex $v$ is mapped to vertex $u$. To every graph $\mathcal{G}=(V_\mathcal{G},E_\mathcal{G})$ we assign a set of clauses $CNF(\mathcal{G},\mathcal{H})$ over the variables in $V(V_\mathcal{G}, V_\mathcal{H})$ in such a way that there is a one-to-one correspondence between the truth valuations of the variables in $V(V_\mathcal{G}, V_\mathcal{H}$) satisfying this set and the homomorphisms from $\mathcal{G}$ to $\mathcal{H}$:

\begin{definition}\label{DefinitionCNF} For any two graphs $\mathcal{G}=(V_\mathcal{G},E_\mathcal{G})$, $\mathcal{H}=(V_\mathcal{H},E_\mathcal{H})$ by $CNF(\mathcal{G},\mathcal{H})$ we denote the following set of clauses: 

\begin{itemize}
    \item a clause $\bigvee_{u\in V_\mathcal{H}}^{} x_{v,u}$ for each $v\in V_\mathcal{G}$;
    \item a clause $\neg x_{v,u_1}\vee\neg x_{v,u_2}$ for each $v\in V_\mathcal{G}$ and $u_1,u_2 \in V_\mathcal{H}$ with $u_1\neq u_2$;
    \item a clause $\neg x_{v_1,u_1}\vee \neg x_{v_2,u_2}$ for every adjacent vertices $v_1,v_2\in V_\mathcal{G}$ and non-adjacent vertices $u_1,u_2 \in V_\mathcal{H}$.
\end{itemize}
It is easy to see that if we exchange the last item with more general definition:
\begin{itemize}
\item a clause $\bigvee_{i\in [r]}^{} \neg x_{v_i,u_i}$ for each natural number $r$, each relation $R$ of arity $r$, each $(v_1,v_2,...,v_r)\in R^\mathcal{G}$, and each $u_1.u_2,...,u_r\notin R^\mathcal{H}$,
\end{itemize}
we get the set of clauses $CNF(\mathcal{G},\mathcal{H})$ for common CSP on any relational structure. 
\end{definition}

\subsection{Bounded Arithmetic}
Some definitions, examples and results are adapted from \cite{10.5555/1734064}. In our work we use \emph{two-sorted first-order} (sometimes called second-order) set-up as a framework for the theory. Here there are two kinds of variables: the variables $x,y,z,...$ of the first sort  are called \emph{number variables} and range over the natural numbers, and the variables $X,Y,Z,...$ of the second sort are called \emph{set (or also strings) variables} and range over finite subsets of natural numbers (which represent binary strings). Functions and predicate symbols may involve both sorts and there are two kinds of functions: the number-valued functions (or just \emph{number functions}) and the string-valued functions (or just \emph{string functions}).

The usual language of arithmetic for two-sorted first-order theories is the extension of standard language for Peano Arithmetic $\mathcal{L_{PA}}$. 
\begin{definition}[$\mathcal{L}^2\mathcal{_{PA}}$]
$\mathcal{L}^2\mathcal{_{PA}} = \{0,1,+,\cdot,\lvert\,    \rvert;=_1,=_2,\leq, \in\}$ 
\end{definition}
Here the symbols $0,1,+,\cdot,=_1$ and $\leq$ are well-known and are from $\mathcal{L_{PA}}$: they are function and predicate symbols over the first sort. The function $\lvert X\rvert$ (\emph{the length of $X$}) is a number-valued function and is intended to denote the least upper bound of the set $X$ (the length of the corresponding string). The binary predicate $\in$ for a number and a set denotes set membership, and $=_2$ is the equality predicate for sets. The defining properties of all symbols from language $\mathcal{L}^2\mathcal{_{PA}}$ are described by a set of basic axioms denoted as 2-$BASIC$ \cite{10.5555/1734064}, which we do not present here.
\begin{notation}
We will use the abbreviation:
$$
X(t) =_{def} t \in X
$$
where $t$ is a number term. Thus we think of $X(i)$ as the $i$-th bit of binary string $X$ of length $\lvert X\rvert$.
\end{notation}

To define the theory $V^0$, in which we will formalize the $\mathcal{H}$-coloring problem, we need the following definitions: 

\begin{definition}[Bounded formulas]
Let $\mathcal{L}$ be a two-sorted vocabulary. If $x$ is a number variable, $X$ is a string variable that do not occur in the $\mathcal{L}$-number term $t$, then $\exists x \leq t \phi$ stands for $\exists x( x \leq t \wedge \phi)$, $\forall x \leq t \phi$ stands for $\forall x(x \leq t \to \phi)$, $\exists X \leq t\phi$ stands for $\exists X(|X|\leq t \wedge \phi)$ and $\forall X\leq t \phi$ stands for $\forall X(|X|\leq t \to \phi)$. Quantifiers that occur in this form are said to be \emph{bounded}, and a \emph{bounded formula} is one in which every quantifier is bounded.
\end{definition}

\begin{definition}[$\Sigma_i^{B}$ and $\Pi_i^{B}$ formulas in $\mathcal{L}^2\mathcal{_{PA}}$]
We will define $\Sigma_i^{B}$ and $\Pi_i^{B}$ formulas recursively as follows:

\begin{itemize}
    \item $\Sigma_0^{B}=\Pi_0^{B}$ is the set of $\mathcal{L}^2\mathcal{_{PA}}$-formulas whose only quantifiers are bounded number quantifiers (there can be free string variables);
    \item For $i\geq 0$, $\Sigma_{i+1}^{B}$ (resp. $\Pi_{i+1}^{B}$) is the set of formulas of the form $\exists \bar{X}\leq \bar{t}\phi(\bar{X})$ (resp. $\forall \bar{X}\leq \bar{t}\phi(\bar{X})$), where $\phi$ is a $\Pi_i^{B}$ formula (resp. $\Sigma_i^{B}$ formula), and $\bar{t}$ is a sequence of $\mathcal{L}^2\mathcal{_{PA}}$-terms not involving any variable from $\bar{X}$.
    
\end{itemize}
\end{definition}

\begin{definition}[Comprehension Axiom]
If $\Phi$ is a set of formulas, then the \emph{comprehension axiom scheme} for $\Phi$, denoted by $\Phi$-$COMP$, is the set of formulas 
\begin{equation}
\exists X \leq y \forall z < y (X(z) \longleftrightarrow \phi(z))
\end{equation}
where $\phi(z)$ is any formula in $\Phi$, $X$ does not occur free in $\phi(z)$, and $\phi(z)$ may have free variables of both sorts, in addition to $z$.
\end{definition}

\begin{definition}[$V^0$] \emph{The theory $V^0$} has the vocabulary $\mathcal{L}^2\mathcal{_{PA}}$ and is axiomatized by $2$-$BASIC$ and $\Sigma^B_{0}$-$COMP$. 

\end{definition}
There is no explicit Induction axiom scheme in $V^0$, but it is known \cite{10.1145/800116.803756} that $V^0\vdash$ $\Sigma^B_{0}$-$IND$, where $\Phi$-$IND$ is:

\begin{definition}[Number Induction Axiom]
If $\Phi$ is a set of two-sorted formulas, then $\Phi$-$IND$ axioms are the formulas
\begin{equation}
(\phi(0)\wedge \forall x (\phi(x)\to \phi(x+1))) \to \forall z\phi(z)
\end{equation}
where $\phi$ is a formula in $\Phi$.
\end{definition}

\subsection{Propositional Proof Complexity}
In this section we define propositional proof systems $R$, $R(log)$ and their tree-like versions. Some definitions and results are adopted from \cite{krajicek_1995},\cite{krajicek2019proof}.

\begin{definition}[Propositional proof system, \cite{cook_reckhow_1979}]
A \emph{propositional proof system} is a polynomial time function $P$ whose range is set $TAUT$. For a tautology $\tau\in TAUT$, any string $w$ such that $P(w)=\tau$ is called a $P$-proof of $\tau$.
\end{definition}
Proof systems are usually defined by a finite number of inference rules of a particular form and the proof is created by applying them step by step. The complexity of proof is measured by its size and number of steps. 

The \emph{resolution system} $R$ operates with atoms and their negations and has no other logical connectives. The basic object is a \emph{clause}, a disjunction of a finite set of literals. The \emph{resolution rule} allows us to derive new clause $C_1\cup C_2$ from two clauses $C_1\cup \{p\}$ and $C_2\cup \{\neg p\}$:
\begin{equation}
\frac{C_1\cup \{p\}\,\,\,\,\,\, C_2\cup \{\neg p\}}{C_1\cup C_2}
\end{equation}
If we manage to derive the \emph{empty clause $\emptyset$} from the initial set of clauses $\mathcal{C}$, the clauses in the set $\mathcal{C}$ are not simultaneously satisfiable. Thus, the resolution system can be interpreted as a \emph{refutation proof system}: instead of proving that a formula is a tautology, it proves that a set of clauses $\mathcal{C}=\{C_1,C_2,...,C_n\}$ is not satisfiable, and therefore the formula $\alpha = \bigvee_{i=1}^{n}\neg C_i$ is a tautology. 

\begin{definition}[An $R$-proof]
Let $\mathcal{C}$ be a set of clauses, an \emph{$R$-refutation} of $\mathcal{C}$ is a sequence of clauses $D_1,...,D_k$ such that:
\begin{itemize}
    \item For each $i\leq k$, either $D_i\in \mathcal{C}$ or there are $u,v< i$ such that $D_i$ follows from $D_u,D_v$ by the resolution rule,
    \item $D_k = \emptyset$.
\end{itemize}
The number of steps in the refutation is $k$.
\end{definition}

The \emph{DNF-resolution} (denoted by DNF-$R$) is a proof system extending $R$ by allowing in clauses not only literals but their conjunctions as well \cite{krajicek2019proof}. DNF-$R$ has the following inference rules:
\begin{equation}
\frac{C\cup\{\bigwedge_{j} l_j\}\,\,\,\,\, D\cup\{\neg l_1',...,\neg l_{t}'\}}{C\cup D}
\end{equation}
if $t\geq 1$ and all $l_i'$ occur among $l_j$, and
\begin{equation}
\frac{C\cup\{\bigwedge_{j\leq s} l_j\}\,\,\,\,\, D \cup \{ \bigwedge_{s<j \leq t} l_j \}}{ C\cup D\cup \{\bigwedge_{i\leq s+t}l_i\}}.
\end{equation}
Notice, that the constant-depth Frege systems generalize the resolution and DNF-$R$ systems, which are depth one and depth two systems respectively.

Let $f:\mathbb{N}^{+}\to \mathbb{N}^{+}$ be a non-decreasing function. Define the $R(f)$-size of a DNF-$R$ refutation $\pi$ to be the minimum $s$ such that:
\begin{itemize}
    \item $\pi$ has at most $s$ steps (that is clauses), and
    \item every logical term occurring in $\pi$ has size at most $f(s)$.
\end{itemize}
Thus, a size $s$ $R(log)$-refutation may contain terms of the size up to $log(s)$.

\begin{definition}[Tree-like proof systems]
A proof is called \emph{tree-like} if every step of the proof is a part of the hypotheses of at most one inference in the proof (each line in the proof can be used only once as hypothesis for an inference rule). For a proof system $P$ by $P^{*}$ we denote the proof system whose proofs are exactly tree-like $P$-proofs, for example $R^*$ and $R^*(log)$.
\end{definition}
\begin{lemmach}[5.7.2 in \cite{krajicek2019proof}]\label{RPSIMRLOG}
$R$ $p$-simulates $R^*(log)$ with respect to refutations of sets of clauses.
\end{lemmach}

We also introduce Definition \ref{DNF1}, which we will use at the end of Sec. 3.3: 
\begin{definition}[DNF$_1$-formula]\label{DNF1} A \emph{basic formula} is an atomic formula or the negation of an atomic formula. A \emph{DNF$_1$-formula} is a formula that is built from basic formulas by:
\begin{itemize}
    \item first apply any number of conjunctions and bounded universal quantifiers,
    \item then apply any number of disjunctions and bounded existential quantifiers.
\end{itemize}

\end{definition}\color{black}

\section{Formalization of the $\mathcal{H}$-coloring problem in $V^0$}
\subsection{Defining Relations}
In this section we define all the notions we need to formalize the decision algorithm in the $p$-time case of the $\mathcal{H}$-coloring problem, i.e. the notions of a graph, bipartite and non-bipartite graphs and a homomorphism between graphs, in the vocabulary $\mathcal{L}^2\mathcal{_{PA}}$ and using only basic axioms of $V^0$. To do this we extend our theory with new predicate and function symbols and for each of them we add defining axioms which ensure that they receive their standard interpretations in a model of $V^0$. 

\begin{definition}[Representable/Definable relations]
Let $\mathcal{L}\supseteq $ $\mathcal{L}^2\mathcal{_{PA}}$ be a two-sorted vocabulary, and let $\phi$ be a $\mathcal{L}$-formula. Then we say that $\phi(\bar{x},\bar{X})$ represents (or defines) a relation $R(\bar{x},\bar{X})$ if 
\begin{equation}
R(\bar{x},\bar{X}) \longleftrightarrow \phi(\bar{x},\bar{X}).
\end{equation}
If $\Phi$ is a set of $\mathcal{L}$-formulas, then we say that $R(\bar{x},\bar{X})$ is \emph{$\Phi$-representable} (or $\Phi$-definable) if it is represented by some $\phi \in \Phi$.
\end{definition}

\begin{definition}[Definable number functions]
Let $T$ be a theory with two-sorted vocabulary $\mathcal{L}\supseteq $ $\mathcal{L}^2\mathcal{_{PA}}$, and let $\Phi$ be a set of $\mathcal{L}$-formulas. A number function $f$ is \emph{$\Phi$-definable} in $T$ if there is a formula $\phi(\bar{x},y,\bar{X})$ in $\Phi$ such that 
\begin{equation}\label{EQDEFINABILITYOFFUNCTION}
T \vdash \forall\bar{x}\forall\bar{X}\,\exists !y\, \phi(\bar{x},y,\bar{X})
\end{equation}
and 
\begin{equation}
y = f(\bar{x},\bar{X}) \longleftrightarrow \phi(\bar{x},y,\bar{X}).
\end{equation}
\end{definition}

Auxiliary predicate and function symbols, which we will use further to define different notions in $V^0$, are the following: 

\begin{definition}[Divisibility] The relation of \emph{divisibility} is defined by: 
\begin{equation}
x|y \longleftrightarrow \exists z \leq y (x\cdot z = y).
\end{equation}
\end{definition}

\begin{definition}[Pairing function]
If $x,y \in \mathbb{N}$ we define the \emph{pairing function} $\langle x,y\rangle$ to be the following term in $V^0$:

\begin{equation}
\langle x,y\rangle = (x + y)(x + y + 1) + 2y
\end{equation}
Since the formula for pairing function is just a term in standard vocabulary for the theory $V^0$, it is obvious that $V^0$ proves the condition  (\ref{EQDEFINABILITYOFFUNCTION}). It is also easy to prove in $V^0$ that pairing function is a one-one function, that is:

\begin{equation}
V^0 \vdash \forall x_1,x_2,y_1,y_2 \ \langle x_1,y_1\rangle = \langle x_2,y_2\rangle \to x_1=x_2 \wedge y_1=y_2
\end{equation}
\end{definition}
Using pairing function we can code pair of numbers $x,y$ by one number $\langle x,y\rangle$, and the sequence of pairs by a subset of numbers. To define a graph on $n$ vertices, consider a string $V_\mathcal{G}$ where $|V_\mathcal{G}|=n$ and $\forall i < n$ $V_\mathcal{G}(i)$. We say that $V_\mathcal{G}$ is the set of $n$ vertices of graph $\mathcal{G}$. Then we define string $E_\mathcal{G}$ of length $|E_\mathcal{G}| < 4n^2$ to be the set of edges of the graph $\mathcal{G}$ as following: if there is an edge between vertices $i,j$ then, using the pairing function, set $E_\mathcal{G}(\langle i,j\rangle)$ and $\neg E_\mathcal{G}(\langle i,j\rangle)$ otherwise. 

\begin{notation}
Instead of $E_\mathcal{G}(\langle i,j\rangle)$ we will write just $E_\mathcal{G}(i,j)$ to denote that there is an edge between $i$ and $j$, and sometimes instead of $(V_\mathcal{G},E_\mathcal{G})$ we will write $\mathcal{G}$.
\end{notation}

\begin{definition}[Undirected graph $\mathcal{G}$ without loops] A pair of sets $\mathcal{G}=(V_\mathcal{G},E_\mathcal{G})$ with $|V_\mathcal{G}|=n$ denotes an undirected graph without loops if it satisfies the following relation:
\begin{equation}
\label{DefinitionUndirectedGraph}
 \begin{split}
&GRAPH(V_\mathcal{G},E_\mathcal{G}) \longleftrightarrow 
 \forall i <n(V_\mathcal{G}(i))\wedge \forall i<j<n \\&\,\,\,\,\,\,\,\,\,\, (E_\mathcal{G}(i,j)\longleftrightarrow E_\mathcal{G}(j,i)) \wedge \forall i<n \,\neg(E_\mathcal{G}(i,i))
 \end{split}
\end{equation}

\end{definition}
Further, talking about graphs we will consider only pairs of strings $\mathcal{G}=(V_\mathcal{G},E_\mathcal{G})$ that satisfy the above relation. Since we formalize the $\mathcal{H}$-coloring problem we need to define the homomorphism on graphs in the vocabulary $\mathcal{L}^2\mathcal{_{PA}}$. Consider two graphs $\mathcal{G}=(V_\mathcal{G},E_\mathcal{G})$ and $\mathcal{H}=(V_\mathcal{H},E_\mathcal{H})$, where $|V_\mathcal{G}|=n$, $|V_\mathcal{H}|=m$. Firstly we define a map between two sets of vertices $V_\mathcal{G},V_\mathcal{H}$, that is between sets $[0,n-1]$ and $[0,m-1]$. We again use the pairing function: consider a set $Z <$ $\langle n-1,m-1\rangle +1$, where $Z(\langle i,j\rangle)$ means that $i$-th vertex is mapped to $j$-th vertex. For $Z$ to be a well-defined map it should satisfy the following $\Sigma_0^{B}$-definable relation $MAP(n,m,Z)$:

\begin{definition}[Map between two sets] We say that a set $Z$ is a \emph{well-defined map} between two sets $[0,n-1]$ and $[0,m-1]$ if it satisfies the relation:
\begin{equation}
 \begin{split}
&\,\,\,\,\,MAP(n,m,Z) \longleftrightarrow \forall i<n \exists j<m\, Z(\langle i,j\rangle)\wedge \\& 
\forall i<n \forall j_1,j_2 < m (Z(\langle i,j_1\rangle)\wedge Z(\langle i,j_2\rangle) \to j_1=j_2)
 \end{split}
\end{equation}
\end{definition}

Now we can formalize the standard notion of existence of a homomorphism between two graphs $\mathcal{G}$ and $\mathcal{H}$ (here the homomorphism is formalized by a set $Z$ with certain properties):
\begin{definition}[The existence of a homomorphism between graphs $\mathcal{G}$ and $\mathcal{H}$] There is a \emph{homomorphism between two graphs $\mathcal{G}=(V_\mathcal{G},E_\mathcal{G})$ and $\mathcal{H}=(V_\mathcal{H},E_\mathcal{H})$} with $\lvert V_\mathcal{G}\rvert = n$, $\lvert V_\mathcal{H}\rvert = m$, if they satisfy the relation:
\begin{equation}
 \begin{split}
&HOM(\mathcal{G},\mathcal{H}) \longleftrightarrow \exists Z \leq \langle n-1,m-1\rangle \big(MAP(n,m,Z)\wedge \\
&\hspace{80pt} \forall i_1,i_2<n, \forall j_1,j_2 < m \\&\hspace{20pt} (E_\mathcal{G}(i_1,i_2)\wedge Z(\langle i_1,j_1\rangle) \wedge Z(\langle i_2,j_2\rangle) \to E_\mathcal{H}(j_1,j_2))\big) 
 \end{split}
\end{equation}
\end{definition}
Note that the relation $HOM(\mathcal{G},\mathcal{H})$ is a $\Sigma^B_1$-definable relation. 

Finally, we need to formalize what does it mean to be a bipartite or a non-bipartite graph. The notion of being bipartite is $\Sigma_1^{B}$-definable in $\mathcal{L}^2\mathcal{_{PA}}$:
\begin{definition}[Bipartite graph $\mathcal{H}$] A graph $\mathcal{H}=(V_\mathcal{H},E_\mathcal{H})$ with $\lvert V_\mathcal{H}\rvert = m$ is \emph{bipartite} if it satisfies the relation:
\begin{equation}
\label{DefinitionBipartite}
 \begin{split}
&\,\,BIP(\mathcal{H})\longleftrightarrow \exists W_\mathcal{H},U_\mathcal{H} \leq m \big(\forall i<m (W_\mathcal{H}(i)\leftrightarrow \neg U_\mathcal{H}(i))\wedge\\
&\forall i<j<m (E_\mathcal{H}(i,j) \to (W_\mathcal{H}(i)\wedge U_\mathcal{H}(j))\vee(W_\mathcal{H}(j)\wedge U_\mathcal{H}(i)))\big)
 \end{split}
\end{equation}
\end{definition}

To define a non-bipartite graph we use a commonly-known characterization of non-bipartite graphs (to contain an odd cycle, or, more generally, to allow a homomoprhism from an odd cycle). The reason here is to get a $\Sigma^B_1$-definable relation for a non-bipartite graph. This makes the formula in the main statement from the next section be $\Pi^B_1$, and hence translatable into propositional logic. First we define a cycle.

\begin{definition}[Cycle $\mathcal{C}_k$]\label{CYCLE} A graph $\mathcal{C}_k=(V_{\mathcal{C}_k},E_{\mathcal{C}_k})$ with $V_{\mathcal{C}_k}=\{0,1,...,k-1\}$ is a \emph{cycle of length $k$} if it satisfies the relation: 
\begin{equation}
 \begin{split}
& CYCLE(\mathcal{C}_k) \longleftrightarrow E_{\mathcal{C}_k}(0,k-1) \wedge \forall i<(k-1)\, E_{\mathcal{C}_k}(i,i+1) \wedge\\ &\hspace{40pt}\forall i,j<(k-1) (j\neq i+1 \to \neg E_{\mathcal{C}_k}(i,j))
 \end{split}
\end{equation}
\end{definition}

\begin{definition}[Non-bipartite graph $\mathcal{G}$]\label{DNBG} A graph $\mathcal{G}=(V_\mathcal{G},E_\mathcal{G})$ with $\lvert V_\mathcal{G}\rvert=n$ is \emph{non-bipartite} if it satisfies the following $\Sigma^B_1$-definable relation: 
\begin{equation}
 \begin{split}
&NONBIP(\mathcal{G}) \longleftrightarrow \exists k\leq n (2|(k-1)) \exists V_{\mathcal{C}_k} = k,\exists E_{\mathcal{C}_k} <4k^2 \\ &\hspace{35pt}CYCLE(V_{\mathcal{C}_k},E_{\mathcal{C}_k})\wedge HOM(\mathcal{C}_k, \mathcal{G})
 \end{split}
\end{equation}
\end{definition}

\subsection{Proving in the theory $V^0$}

\begin{lemmach}[Homomorphism transitivity]\label{HOMTRAMS} For all graphs $\mathcal{G},\mathcal{H},\mathcal{S}$
\begin{equation}
V^0\vdash\forall \mathcal{G}, \mathcal{H}, \mathcal{S} \,( HOM(\mathcal{G},\mathcal{H})\wedge HOM(\mathcal{H},\mathcal{S})  \to  HOM(\mathcal{G},\mathcal{S}))
\end{equation}
\end{lemmach}

\begin{proof}
Consider graphs $\mathcal{G}=(V_\mathcal{G},E_\mathcal{G}))$, $\mathcal{H}=(V_\mathcal{H},E_\mathcal{H})$ and $\mathcal{S}=(V_\mathcal{S},E_\mathcal{S})$, where $|V_\mathcal{G}|=n$, $|V_\mathcal{H}|=m$ and $|V_\mathcal{S}|=t$. Since $HOM(\mathcal{G},\mathcal{H})$ and $HOM(\mathcal{H},\mathcal{S})$, there exist two sets $Z \leq$ $\langle n-1,m-1\rangle$ and $Z' \leq$ $\langle m-1,t-1\rangle$ which satisfy the homomorphism definition. We need to prove that there exists a set $Z'' \leq$ $\langle n-1,t-1\rangle$, such that: 
\begin{align*}
&\hspace{20pt}MAP(n,t,Z'') \wedge \forall i_1,i_2 < n, \forall k_1,k_2< t \\
&(E_\mathcal{G}(i_1,i_2) \wedge Z''(\langle i_1,k_1\rangle)\wedge Z''(\langle i_2,k_2\rangle) \to E_\mathcal{S}(k_1,k_2))
\end{align*}
Consider the set $Z'' \leq$ $\langle n-1,t-1\rangle$ which we define by the formula: 
\begin{equation}
Z''(\langle i,k\rangle) \longleftrightarrow \exists j<m(Z(\langle i,j\rangle)\wedge Z'(\langle j,k\rangle)).
\end{equation}
This set should exist due to Comprehension Axiom $\Sigma_0^{B}$-$COMP$, since the formula $\phi(\langle i,k\rangle) = \exists j<m$ $(Z(\langle i,j\rangle)\wedge Z'(\langle j,k\rangle)) \in \Sigma_0^{B}$. It is easy to check that the set $Z''$ satisfies the homomorphism relation between graphs $\mathcal{G}$ and $\mathcal{S}$.
\end{proof}

\begin{notation}
$K_2$ will denote the complete graph on two vertices.
\end{notation}
In the following two lemmas we prove that there is always a homomorphism from a bipartite graph to $K_2$ and there is no homomorphism from a non-bipartite graph to $K_2$.

\begin{lemmach}\label{MTL2} For all bipartite graphs $\mathcal{H}$, $V^0$ proves the existence of a homomorphism from $\mathcal{H}$ to $\mathcal{K}_2$:
\begin{equation}
V^0\vdash\forall \mathcal{H}\, (BIP(\mathcal{H})\to  HOM(\mathcal{H},\mathcal{K}_2))
\end{equation}
\end{lemmach}

\begin{proof} Consider a bipartite graph $\mathcal{H}=(V_\mathcal{H},E_\mathcal{H})$ with $|V_\mathcal{H}|=n$. We need to show that there exists a homomorphism from $\mathcal{H}$ to $\mathcal{K}_2$, that is an appropriate set $Z \leq$ $\langle n-1,2\rangle$. Since $\mathcal{H}$ is bipartite, then there exist two subsets $W_\mathcal{H}$ and $U_\mathcal{H}$, such that $(W_\mathcal{H}(i)\leftrightarrow \neg U_\mathcal{H}(i))$. Consider a set $Z \leq$ $\langle n-1,2\rangle$, such that:
$$ 
\left\{ 
\begin{array}{rcl} 
Z(\langle i,0\rangle)\longleftrightarrow W_\mathcal{H}(i)\\
Z(\langle i,1\rangle) \longleftrightarrow U_\mathcal{H}(i)
\end{array}
\right. 
$$
This set also exists due to Comprehension Axiom $\Sigma_0^{B}$-$COMP$, since the formula $\phi(\langle i,j\rangle) =  (j=0\wedge W_\mathcal{H}(i)) \vee (j=1\wedge U_\mathcal{H}(i)) \in \Sigma_0^{B}$. Obviously, since $(W_\mathcal{H}(i)\leftrightarrow \neg U_\mathcal{H}(i))$, by the definition of $Z$ we have $MAP(n,2,Z)$. Consider any $i_1,i_2 <n$, such that $E_\mathcal{H}(i_1,i_2)$. Then $(W_\mathcal{H}(i_1)\wedge U_\mathcal{H}(i_2))$ or $(W_\mathcal{H}(i_2)\wedge U_\mathcal{H}(i_1))$. In the first case we have $Z(\langle i_1,0\rangle)\wedge Z(\langle i_2,1\rangle)$, in the second case $Z(\langle i_2,0\rangle )\wedge Z(\langle i_1,1\rangle )$, and in both cases $E_{\mathcal{K}_2}(0,1)$. Thus, $Z$ is a homomorphism from $\mathcal{H}$ to $\mathcal{K}_2$.
\end{proof}

\begin{lemmach}\label{MTL3} For all non-bipartite graphs $\mathcal{G}$, $V^0$ proves that there is no homomorphism from $\mathcal{G}$ to $\mathcal{K}_2$:
\begin{equation}\label{NONK2}
V^0\vdash\forall \mathcal{G} \,\,(NONBIP(\mathcal{G})\to  \neg HOM(\mathcal{G},\mathcal{K}_2))
\end{equation}
\end{lemmach}

\begin{proof}
Suppose that a graph $\mathcal{G}=(V_\mathcal{G},E_\mathcal{G})$, $\lvert V_\mathcal{G}\rvert = n$ is non-bipartite, that is there exist $k\leq n$, $\mathcal{C}_k = (V_{\mathcal{C}_k}, H_{\mathcal{C}_k})$ with $\lvert V_{\mathcal{C}_k}\rvert = k$, such that $2|(k-1)$, $CYCLE(\mathcal{C}_k)$ and $HOM(\mathcal{C}_k, \mathcal{G})$.

Assume that there exists a homomorphism form $\mathcal{G}$ to $\mathcal{K}_2$. Due Lemma \ref{HOMTRAMS} by transitivity there also exists a homomorphism $Z \leq \langle k-1,2\rangle$ from $\mathcal{C}_k$ to $\mathcal{K}_2$. Since it is a homomorphism from $\mathcal{C}_k$ to $\mathcal{K}_2$ then for every $0 \leq i \leq (k-1)$ either $Z(\langle i,0\rangle )$ or $Z(\langle i,1\rangle )$. 

Without loss of generality suppose that $Z(\langle 0,0\rangle )$ and lets prove that \\ $Z(\langle k-1,0\rangle )$ too. Since $2|(k-1)$, then $k>2$. Due $CYCLE(\mathcal{C}_k)$, $E_{\mathcal{C}_k}(0,1)$ and $E_{\mathcal{C}_k}(1,2)$. We claim that for every $i<k$, if $2|i$ then $Z(\langle i,0\rangle )$ and  $Z(\langle i,1\rangle )$ otherwise. Consider the formula:  
\begin{equation}\label{INDLEMMA}
\phi(i,k,Z) = (2|i \to Z(\langle i,0\rangle ))\wedge (2\nmid i \to Z(\langle i,1\rangle ))
\end{equation}
Since $\phi(i,n,Z) \in \Sigma_0^{B}$, we can prove this claim by induction on $i$, because $V^0$ proves $\Sigma_0^{B}$-$IND$:
\begin{equation}
(\phi(0,n,Z)\wedge \forall i<k (\phi(i,n,Z) \to \phi(i+1,n,Z)) \to \forall j<k \, \phi(j,n,Z)
\end{equation}
The base case is considered above. For step of induction suppose that it is true for $(i-1)$ and consider $i$. We have two options. If $2|(i-1)$ then by the induction hypothesis $Z(\langle i-1,0\rangle )$. Thus, since for $(i-1)$ by $CYCLE(\mathcal{C}_k)$ we have  $E_{\mathcal{C}_k}(i-1,i)$, by the definition of the homomorphism $Z(\langle i,1\rangle)$. Analogously, if $2\nmid (i-1)$ then $Z(\langle i,0\rangle )$.

Hence $Z(\langle 0,0\rangle )$ and $Z(\langle k-1,0\rangle )$. But since there is an edge between vertices $0$ and $(k-1)$ in the graph $\mathcal{C}_k$, $Z$ cannot be a homomorphism between $\mathcal{C}_k$ and $K_2$. Therefore, our assumption leads to contradiction and there is no homomorphism from $\mathcal{G}$ to $\mathcal{K}_2$.
\end{proof}

The main result of this paper is an immediate conclusion from the previous lemmas.
\begin{theorem}[The main universal statement]\label{MTT1} For all non-bipartite graphs $\mathcal{G}$ and bipartite graphs $\mathcal{H}$, $V^0$ proves that there is no homomorphism from $\mathcal{G}$ to $\mathcal{H}$:
\begin{equation}
V^0\vdash\forall \mathcal{G},\mathcal{H}  (BIP(\mathcal{H})\wedge NONBIP(\mathcal{G})\to  \neg HOM(\mathcal{G},\mathcal{H}))
\end{equation}
\end{theorem}
\begin{proof}
Suppose that there exists a homomorphism from $\mathcal{G}$ to $\mathcal{H}$. According to Lemma \ref{MTL2}, since $\mathcal{H}$ is bipartite then there exists a homomorphism from $\mathcal{H}$ to $K_2$. Thus due to Lemma \ref{HOMTRAMS} by the transitivity there exists a homomorphism from $\mathcal{G}$ to $K_2$. But this is the contradiction with Lemma \ref{MTL3}.
\end{proof}

\subsection{Translating into tautologies}
\subsubsection{Translation of the main universal statement}
In this section we proceed with translation of the main universal statement in the theory $V^0$ into propositional tautologies. There is a well-known translation of $\Sigma_0^B$ formulas into propositional calculus formulas: we can translate each formula $\phi(\bar{x},\bar{X}) \in \Sigma_0^B$ into a family of propositional formulas \cite{10.5555/1734064}:
\begin{equation}
||\phi(\bar{x},\bar{X})|| = \{\phi(\bar{x},\bar{X})[\bar{m},\bar{n}]: \bar{m},\bar{n} \in \mathbb{N}\}
\end{equation}

\begin{lemmach}[\cite{10.5555/1734064}]\label{LPropositionalTranslation} For every $\Sigma_0^B(\mathcal{L}^2\mathcal{_{PA}})$ formula $\phi(\bar{x},\bar{X})$, there is a constant $d\in \mathbb{N}$ and a polynomial $p(\bar{m},\bar{n})$ such that for all $\bar{m},\bar{n} \in \mathbb{N}$, the propositional formula $\phi(\bar{x},\bar{X})[\bar{m},\bar{n}]$ has depth at most $d$ and size at most $p(\bar{m},\bar{n})$ \cite{10.5555/1734064}.
\end{lemmach}
There is a theorem that establish a connection between $\Sigma_0^B$-fragment of the theory $V^0$ and constant-depth Frege proof system:
\begin{theorem}[$V^0$ Translation, \cite{10.5555/1734064}] \label{Translation} Suppose that $\phi(\bar{x},\bar{X})$ is a $\Sigma_0^B$ formula such that $V^0 \vdash \forall\bar{x}\forall\bar{X}\phi(\bar{x},\bar{X})$. Then the propositional family $||\phi(\bar{x},\bar{X})||$ has polynomial size bounded depth Frege proofs. That is, there are a constant $d$ and a polynomial $p(\bar{m},\bar{n})$ such that for all $1\leq \bar{m},\bar{n}\in \mathbb{N}$, $\phi(\bar{x},\bar{X})[\bar{m},\bar{n}]$ has a $d$-Frege proof of size at most $p(\bar{m},\bar{n})$. Further there is an algorithm which finds a $d$-Frege proof of $\phi(\bar{x},\bar{X})[\bar{m},\bar{n}]$ in time bounded by a polynomial in $(\bar{m},\bar{n})$ \cite{10.5555/1734064}.
\end{theorem}

Consider the $\Pi_1^B$-formula $\phi(\mathcal{G},\mathcal{H})$ from Theorem \ref{MTT1} which expresses that there is no homomorphism from a non-bipartite graph $\mathcal{G}$ to a bipartite graph $\mathcal{H}$:
\begin{equation}\label{OURFORMULA}
 \begin{split}
&\phi(\mathcal{G},\mathcal{H}) = \neg GRAPH(\mathcal{G})\vee\neg GRAPH(\mathcal{H})\vee\\
&\neg BIP(\mathcal{H})\vee \neg NONBIP(\mathcal{G})\vee  \neg HOM(\mathcal{G},\mathcal{H})
 \end{split}
\end{equation}
For the graphs $\mathcal{G}=(V_\mathcal{G},E_\mathcal{G})$ with $|V_\mathcal{G}|=n$ and $\mathcal{H}=(V_\mathcal{H},E_\mathcal{H})$ with $|V_\mathcal{H}|=m$ we can rewrite this formula as follows:

\begin{align*}
&\phi(V_\mathcal{G},E_\mathcal{G},V_\mathcal{H},E_\mathcal{H}) = \\\\
&\exists i< n \,\neg V_\mathcal{G}(i)\vee \exists i<j<n ((\neg E_\mathcal{G}(i,j)\vee \neg E_\mathcal{G}(j,i))&\text{(I)}\\
&\wedge(E_\mathcal{G}(i,j)\vee E_\mathcal{G}(j,i))) \vee \exists i<n \,E_\mathcal{G}(i,i) \\
&\hspace{100pt}\vee \\
&\exists i< m \,\neg V_\mathcal{H}(i)\vee \exists i<j<m ((\neg E_\mathcal{H}(i,j)\vee \neg E_\mathcal{H}(j,i))&\text{(II)}\\
&\wedge(E_\mathcal{H}(i,j)\vee E_\mathcal{H}(j,i))) \vee \exists i<n \,E_\mathcal{H}(i,i) \\
&\hspace{100pt}\vee \\
&\forall W_\mathcal{H},U_\mathcal{H} \leq m \big(\exists i<m\, ((\neg W_\mathcal{H}(i)\vee U_\mathcal{H}(i))\wedge(W_\mathcal{H}(i)\vee \neg U_\mathcal{H}(i))) \vee&\text{(III)}\\
&\exists i<j<m (E_\mathcal{H}(i,j) \wedge (\neg W_\mathcal{H}(i)\vee \neg U_\mathcal{H}(j))\wedge(\neg W_\mathcal{H}(j)\vee \neg U_\mathcal{H}(i)))\big) \\
&\hspace{100pt}\vee \\
&\forall k\leq n (2|(k-1))\, \forall V_{\mathcal{C}_k} = k,\, \forall E_{\mathcal{C}_k} < 4k^2 \big(  (\exists i<k \, \neg V_{\mathcal{C}_k}(i) \vee \\
&\exists i<j<k ((\neg E_{\mathcal{C}_k}(i,j)\vee\neg E_{\mathcal{C}_k}(j,i))\wedge(E_{\mathcal{C}_k}(i,j)\vee E_{\mathcal{C}_k}(j,i)) )      \vee \\
&\exists i<k \, E_{\mathcal{C}_k}(i,i)  ) \vee (   \neg E_{\mathcal{C}_k}(0,k-1)\vee\exists i < (k-1)&\text{(IV)}\\
&\neg E_{\mathcal{C}_k}(i,i+1)  \vee \exists i,j < (k-1) \, (j\neq i+1 \wedge E_{\mathcal{C}_k}(i,j))) \vee \\
& (\forall Z \leq \langle k-1,n-1\rangle \,(\neg MAP(k,n,Z) \vee \exists i_1,i_2<k \exists j_1,j_2 <n \\
&E_{\mathcal{C}_k}(i_1,i_2)\wedge Z(\langle i_1,j_1\rangle)\wedge Z(\langle i_2,j_2\rangle)\wedge \neg E_{\mathcal{G}}(j_1,j_2)          ))       \big)\\
&\hspace{100pt}\vee  \\
&\forall Z' \leq \langle n-1,m-1\rangle \big(\neg MAP(n,m,Z')\vee  \exists i_1,i_2<n, \exists j_1,j_2 < m &\text{(V)}\\
&(E_\mathcal{G}(i_1,i_2)\wedge Z'(\langle i_1,j_1\rangle) \wedge  Z'(\langle i_2,j_2\rangle) \wedge \neg E_\mathcal{H}(j_1,j_2))\big)
\end{align*}
In strict form (with all string quantifiers occur in front) the formula $\phi(V_\mathcal{G},E_\mathcal{G},V_\mathcal{H},E_\mathcal{H})$ looks like:
\begin{equation}\label{MainFormula}
 \begin{split}
&\phi(V_\mathcal{G},E_\mathcal{G},V_\mathcal{H},E_\mathcal{H}) = \forall W_\mathcal{H},U_\mathcal{H} \leq m, \forall V_{\mathcal{C}_k}\leq n, \forall E_{\mathcal{C}_k} \leq 4n^2, \\&\forall Z \leq \langle k-1,n-1\rangle,\forall Z' \leq \langle n-1,m-1\rangle\\
&[\psi(n,m,V_\mathcal{G},V_\mathcal{H},W_\mathcal{H},U_\mathcal{H},V_{\mathcal{C}_k},E_\mathcal{G},E_\mathcal{H},E_{\mathcal{C}_k},Z,Z')],
 \end{split}
\end{equation}
where $\psi(n,m,V_\mathcal{G},V_\mathcal{H},W_\mathcal{H},U_\mathcal{H},V_{\mathcal{C}_k},E_\mathcal{G},E_\mathcal{H},E_{\mathcal{C}_k},Z,Z')$ is the $\Sigma_0^B$-formula. Thus, by Lemma \ref{LPropositionalTranslation} one can translate it into a family of short propositional formulas. For every free string variable $X$, $|X|=n_X$ in the formula $\psi$ we introduce propositional variables $p^X_{0},p^X_{1},...,p^X_{n_{(X-1)}}$ where $p^X_{i}$ is intended to mean $X(i)$. The first two parts (I),(II) of the formula $\phi(V_\mathcal{G},E_\mathcal{G},V_\mathcal{H},E_\mathcal{H})$ say that $\mathcal{G},\mathcal{H}$ are not graphs. Free number variables here are $n,m$, free string variables are $V_\mathcal{G},V_\mathcal{H},E_\mathcal{G},E_\mathcal{H}$. For graph $\mathcal{G}$, (I) translates into: 
\begin{equation}\label{NOTAGRAPH}
 \begin{split}
&\big[\bigvee_{i=0}^{n-1}(\neg p_i^{V_\mathcal{G}})\big]\vee \big[\bigvee_{j=0}^{n-1}\bigvee_{i=0}^{j-1} (\neg p_{\langle i,j\rangle}^{E_\mathcal{G}}\vee \neg p_{\langle j,i\rangle}^{E_\mathcal{G}})\wedge(p_{\langle i,j\rangle}^{E_\mathcal{G}}\vee p_{\langle j,i\rangle}^{E_\mathcal{G}})\big]\vee \\
&\big[\bigvee_{i=0}^{n-1}( p_{\langle i,i\rangle}^{E_\mathcal{G}})\big]
 \end{split}
\end{equation}
And for graph $\mathcal{H}$, (II) translates into:
\begin{equation}
 \begin{split}
&\big[\bigvee_{i=0}^{m-1}(\neg p_i^{V_\mathcal{H}})\big]\vee \big[\bigvee_{j=0}^{m-1}\bigvee_{i=0}^{j-1} (\neg p_{\langle i,j\rangle}^{E_\mathcal{H}}\vee \neg p_{\langle j,i\rangle}^{E_\mathcal{H}})\wedge(p_{\langle i,j\rangle}^{E_\mathcal{H}}\vee p_{\langle j,i\rangle}^{E_\mathcal{H}})\big]\vee \\
&\big[\bigvee_{i=0}^{m-1}( p_{\langle i,i\rangle}^{E_\mathcal{H}})\big]
 \end{split}
\end{equation}
The third part (III) of the formula $\phi(V_\mathcal{G},E_\mathcal{G},V_\mathcal{H},E_\mathcal{H})$ is about the graph $\mathcal{H}$ not being bipartite, free number variable here is $m$, free string variables are $W_\mathcal{H},U_\mathcal{H},E_\mathcal{H}$. The translation of (III) is:  
\begin{equation}
 \begin{split}
&\big[\bigvee_{i=0}^{m-1} (\neg p_i^{W_\mathcal{H}}\vee p_i^{U_\mathcal{H}})\wedge ( p_i^{W_\mathcal{H}}\vee \neg p_i^{U_\mathcal{H}})\big]\vee\\
&\big[\bigvee_{j=0}^{m-1}\bigvee_{i=0}^{j-1}
p_{\langle i,j\rangle}^{E_\mathcal{H}} \wedge(\neg p_i^{W_\mathcal{H}}\vee \neg p_j^{U_\mathcal{H}})\wedge (\neg p_j^{W_\mathcal{H}}\vee \neg p_i^{U_\mathcal{H}})\big]
 \end{split}
\end{equation}
The fourth part (IV) of the formula $\phi(V_\mathcal{G},E_\mathcal{G},V_\mathcal{H},E_\mathcal{H})$ expresses that $\mathcal{G}$ is not a non-bipartite graph. Free number variable here is $n$, free string variables are $V_{\mathcal{C}_k},E_{\mathcal{C}_k},Z$. This complex subformula we split into parts. Firstly, the part of subformula saying that $\mathcal{C}_k$ is not a graph is translated into:
\begin{equation}\label{NOTAGRAPHC}
 \begin{split}
&\big[\bigvee_{i=0}^{k-1}(\neg p_i^{V_{\mathcal{C}_k}})\big]\vee \big[\bigvee_{j=0}^{k-1}\bigvee_{i=0}^{j-1} (\neg p_{\langle i,j\rangle}^{E_{\mathcal{C}_k}}\vee \neg p_{\langle j,i\rangle}^{E_{\mathcal{C}_k}})\wedge(p_{\langle i,j\rangle}^{E_{\mathcal{C}_k}}\vee p_{\langle j,i\rangle}^{E_{\mathcal{C}_k}})\big]\vee \\
&\big[\bigvee_{i=0}^{n-1}( p_{\langle i,i\rangle}^{E_{\mathcal{C}_k}})\big]
 \end{split}
\end{equation}
Then the part saying that $\mathcal{C}_k$ is not a cycle translates into:
\begin{equation}
 \begin{split}
&\big[ \neg p^{E_{\mathcal{C}_k}}_{\langle 0,k-1\rangle} \big]  \vee  \big[\bigvee_{i=0}^{k-2} \neg p^{E_{\mathcal{C}_k}}_{\langle i,i+1\rangle}      \big]   \vee  \big[ \bigvee_{i=0}^{k-2} \bigvee_{j=0, \, j\neq i+1}^{k-2} p^{E_{\mathcal{C}_k}}_{\langle j,i\rangle}    \big]
 \end{split}
\end{equation}
And the part, saying that $Z$ is not a map or not a homomorphism between $\mathcal{C}_k$ and $\mathcal{G}$, is translated into:
\begin{equation}\label{IVLAST}
 \begin{split}
&\big[ \bigvee_{i=0}^{k-1}\bigwedge_{j=0}^{n-1} (\neg p^{Z}_{\langle i,j\rangle}) \big]\vee \big[  \bigvee_{i=0}^{k-1}\bigvee_{j_2=0}^{n-1} \bigvee_{j_1=0,\,j_1\neq j_2}^{n-1}(p^{Z}_{\langle i,j_1\rangle} \wedge   p^{Z}_{\langle i,j_2\rangle})\big] \vee\\
&\big[ \bigvee_{i_1,i_2=0}^{k-1}\bigvee_{j_1,j_2=0}^{n-1} ( p^{E_{\mathcal{C}_k}}_{\langle i_1,i_2\rangle}\wedge p^{Z}_{\langle i_1,j_1\rangle}\wedge p^{Z}_{\langle i_2,j_2\rangle}\wedge \neg p^{E_\mathcal{G}}_{\langle j_1,j_2\rangle})                  \big]
 \end{split}
\end{equation}
Finally, to get the translation of the whole subformula we need first to make a disjunction of all formulas (\ref{NOTAGRAPHC})-(\ref{IVLAST}) and then make a conjunction on $k$:

\begin{equation}
 \begin{split}
&\bigwedge_{k=3,\, 2|(k-1)}^{n-1}\Biggl[
\big[   \bigvee_{i=0}^{k-1}(\neg p_i^{V_{\mathcal{C}_k}})\big]\vee \big[\bigvee_{j=0}^{k-1}\bigvee_{i=0}^{j-1} (\neg p_{\langle i,j\rangle}^{E_{\mathcal{C}_k}}\vee \neg p_{\langle j,i\rangle}^{E_{\mathcal{C}_k}})\wedge(p_{\langle i,j\rangle}^{E_{\mathcal{C}_k}}\vee p_{\langle j,i\rangle}^{E_{\mathcal{C}_k}})\big]\vee \\
&\hspace{20pt}\big[\bigvee_{i=0}^{n-1}( p_{\langle i,i\rangle}^{E_{\mathcal{C}_k}})   \big] \vee \big[ \neg p^{E_{\mathcal{C}_k}}_{\langle 0,k-1\rangle} \big]  \vee  \big[\bigvee_{i=0}^{k-2} \neg p^{E_{\mathcal{C}_k}}_{\langle i,i+1\rangle}      \big]   \vee  \big[ \bigvee_{i=0}^{k-2} \bigvee_{j=0, \, j\neq i+1}^{k-2} p^{E_{\mathcal{C}_k}}_{\langle j,i\rangle}    \big] \vee\\
&\hspace{40pt}\big[ \bigvee_{i=0}^{k-1}\bigwedge_{j=0}^{n-1} (\neg p^{Z}_{\langle i,j\rangle}) \big]\vee \big[  \bigvee_{i=0}^{k-1}\bigvee_{j_2=0}^{n-1} \bigvee_{j_1=0,\,j_1\neq j_2}^{n-1}(p^{Z}_{\langle i,j_1\rangle} \wedge   p^{Z}_{\langle i,j_2\rangle})\big] \vee\\
&\hspace{40pt}\big[ \bigvee_{i_1,i_2=0}^{k-1}\bigvee_{j_1,j_2=0}^{n-1} ( p^{E_{\mathcal{C}_k}}_{\langle i_1,i_2\rangle}\wedge p^{Z}_{\langle i_1,j_1\rangle}\wedge p^{Z}_{\langle i_2,j_2\rangle}\wedge \neg p^{E_\mathcal{G}}_{\langle j_1,j_2\rangle})                  \big]
\Biggr] 
 \end{split}
\end{equation}
And the fifth part (V) of the formula $\phi(V_\mathcal{G},E_\mathcal{G},V_\mathcal{H},E_\mathcal{H})$ saying that there is no homomorphism from $\mathcal{G}$ to $\mathcal{H}$, with free number variables $n,m$, free string variables $Z',E_\mathcal{G},E_\mathcal{H}$, is translated into: 
\begin{equation}\label{NOTAHOM}
 \begin{split}
&\big[ \bigvee_{i=0}^{n-1}\bigwedge_{j=0}^{m-1} (\neg p^{Z'}_{\langle i,j\rangle}) \big]\vee \big[  \bigvee_{i=0}^{n-1}\bigvee_{j_2=0}^{m-1} \bigvee_{j_1=0,\,j_1\neq j_2}^{m-1}(p^{Z'}_{\langle i,j_1\rangle} \wedge   p^{Z'}_{\langle i,j_2\rangle})\big]\vee \\
&\big[ \bigvee_{i_1,i_2=0}^{n-1}\bigvee_{j_1,j_2=0}^{m-1} ( p^{E_\mathcal{G}}_{\langle i_1,i_2\rangle}\wedge p^{Z'}_{\langle i_1,j_1\rangle}\wedge p^{Z'}_{\langle i_2,j_2\rangle}\wedge \neg p^{E_\mathcal{H}}_{\langle j_1,j_2\rangle})                  \big]
 \end{split}
\end{equation}
The family of propositional formulas $||\psi(n,m,V_\mathcal{G},V_\mathcal{H},W_\mathcal{H},U_\mathcal{H},V_{\mathcal{C}_k},E_\mathcal{G},E_\mathcal{H},$ $E_{\mathcal{C}_k},$ $Z,Z')||$ is therefore the disjunction of formulas (\ref{NOTAGRAPH})-(\ref{NOTAHOM}) for all possible $n$, $m$, $n_{V_\mathcal{G}}$, $n_{V_\mathcal{H}}$, $n_{W_\mathcal{H}}$, $n_{U_\mathcal{H}}$, $n_{V_{\mathcal{C}_k}}$, $n_{E_\mathcal{G}}$, $n_{E_\mathcal{H}}$, $n_{E_{\mathcal{C}_k}}$, $n_{Z}$, $n_{Z'} \in \mathbb{N}$. By Theorem \ref{Translation} this family of tautologies has polynomial size bounded depth Frege proof.

We are now ready to prove our main goal, to show that the
formulas \\ $\neg HOM(\mathcal{G},\mathcal{H})$, for any non-bipartite graph $\mathcal{G}$ and bipartite graph $\mathcal{H}$, have short propositional proofs.

\begin{theorem}[The main result] For any non-bipartite graph $\mathcal{G}$ and bipartite graph $\mathcal{H}$ the propositional family $||\neg HOM(\mathcal{G},\mathcal{H})||$ has polynomial size bounded depth Frege proofs.
\end{theorem}

\begin{proof}
By the construction above and Theorem \ref{Translation} the translation of formula (\ref{OURFORMULA}) has $p$-size constant-depth Frege proof. If $\mathcal{G}$ and $\mathcal{H}$ are graphs, then the translations of the first two disjuncts in (\ref{OURFORMULA}) are propositional sentences that evaluate to $0$ and thus can be computed in the proof system.

Further, because $\mathcal{H}$ is bipartite, we can find its two parts $W_\mathcal{H}, U_\mathcal{H}$ and evaluate accordingly the atoms in the translation of $\neg BIP(\mathcal{H})$ corresponding to $W_\mathcal{H}$ and $U_\mathcal{H}$ such that the whole translation of the disjuct $\neg BIP(\mathcal{H})$ becomes false. That is, as before it is a propositional sentence that evaluates to $0$. Analogous argument removes the translation of the disjunct $\neg NONBIP(\mathcal{G})$: substitute for the atoms corresponding to a homomorphism from an odd cycle for some $k$ values determined by an actual homomorphism from $\mathcal{C}_k$ into $\mathcal{G}$. This will turn the translation of the fourth disjunct $\neg NONBIP(\mathcal{G})$ into a sentence equal to $0$ as well.

To summarize: after these substitutions the first four disjucts in the translation of the formula (\ref{OURFORMULA}) become propositional sentences evaluating to $0$ and thus the whole translation of the formula (\ref{OURFORMULA}) is equivalent to the translation of $\neg HOM(\mathcal{G},\mathcal{H})$. That is, we obtained polynomial size constant-depth Frege proof of 
$||\neg HOM(G,H)||$.
\end{proof}

\subsubsection{Other Remarks}
Actually, we can improve a little bit our upper bound result from the Sec. 3.3.1. To reason about graph we used convenient for this purpose set-up of two-sorted theory $V^0$, including the Comprehension axiom. But actually we can avoid using it in both proofs of Lemmas \ref{HOMTRAMS} and \ref{MTL2}. For example, in the proof of Lemma \ref{HOMTRAMS} instead of declaring the existence of the set $Z''(\langle i,k\rangle) \longleftrightarrow \exists j<m(Z(\langle i,j\rangle)\wedge Z'(\langle j,k\rangle))$ by the Comprehension axiom we can derive that there always exists such $j<m$ that $Z(\langle i,j\rangle)$ and $Z'(\langle j,k\rangle)$ (since $MAP(n,m,Z)\wedge MAP(m,t,Z')$) and therefore just manually construct the appropriate set $Z''$. Thus, we can switch between the theory $V^0$ and the weaker theory $I\Sigma^{1,b}_0$, which is axiomatized by $2$-$BASIC$ and the $I\Sigma^{1,b}_0$-$IND$ (where $I\Sigma^{1,b}_0$ denotes the class of $\mathcal{L}^2\mathcal{_{PA}}$-formulas with all number quantifiers bounded and with no set-sort quantifiers) when it is needed. Moreover, we can restrict further the complexity of formulas in the Induction scheme from the full class $I\Sigma^{1,b}_0$ to its subclass $\Sigma^b_{1}$ (which allows only existential number quantifiers bounded) since we use Induction scheme only once for $\Sigma^b_{1}$-formula (\ref{INDLEMMA}) in the proof of Lemma \ref{MTL3}. 

Denote by $T_1^1(\alpha)$ the two-sorted theory in the vocabulary $\mathcal{L}^2\mathcal{_{PA}}$, containing $2$-$BASIC$ and $IND$ scheme for $\Sigma^b_1$-formulas. Then there is a theorem: 
\begin{theorem}[\cite{krajicek2019proof}] 
Suppose that $\phi(\bar{x},\bar{X})$ is a $\Sigma_0^B$, DNF$_1$-formula such that $T_1^1(\alpha) \vdash \forall\bar{x}\forall\bar{X}\phi(\bar{x},\bar{X})$. Then the propositional family $||\phi(\bar{x},\bar{X})||$ has polynomial size $R^*(log)$-proofs. That is, there is a polynomial $p(\bar{m},\bar{n})$ such that for all $1\leq \bar{m},\bar{n}\in \mathbb{N}$, $\phi(\bar{x},\bar{X})[\bar{m},\bar{n}]$ has an $R^*(log)$-refutation of size at most $p(\bar{m},\bar{n})$. Further there is an algorithm which finds an $R^*(log)$-refutation of $\phi(\bar{x},\bar{X})[\bar{m},\bar{n}]$ in time bounded by a polynomial in $(\bar{m},\bar{n})$.
\end{theorem}

It is obvious that we can modify a little the formula (\ref{MainFormula}) to become DNF$_1$. Thus, the negations of the family of tautologies, expressing that there is no homomorphism from a non-bipartite graph $\mathcal{G}$ to a bipartite graph $\mathcal{H}$ have short $R^*(log)$-refutation in $R^*(log)$ system, which is essentially a constant-depth Frege system with depth $2$ and narrow logical terms. 

Another note that one of our auxiliary lemmas, Lemma \ref{MTL3}, gives us a collateral result. The $\Pi_1^B$-formula (\ref{NONK2}):
\begin{align*}
\phi(\mathcal{G}) = \neg NONBIP(\mathcal{G})\vee  \neg HOM(\mathcal{G},\mathcal{K}_2),
\end{align*}
expressing that there is no homomorphism from non-bipartite graph $\mathcal{G}$ to complete graph $\mathcal{K}_2$, also could be rewritten in strict form as the universal statement of the $\Sigma^B_0$-fragment of $V^0$. Thus, the family of tautologies into which one can translate this universal statement also has polynomial size $R^*(log)$-proofs. Essentially, the formula (\ref{NONK2}) means that the sets
of bipartite and non-bipartite graphs are disjoint, since we can define a bipartite graph $\mathcal{H}$ as:
\begin{equation}
BIP(\mathcal{H}) \longleftrightarrow HOM(\mathcal{H},\mathcal{K}_2).
\end{equation}
We know that resolution $R$ $p$-simulates $R^*(log)$ system (see Lemma \ref{RPSIMRLOG}). Thus, due to the feasible interpolation Theorem \ref{INTERPOL}, there is a $p$-time algorithm separating bipartite and non-bipartite graphs. Of course, this is well-known but here we obtain the algorithm as a consequence of the existence of short resolution proofs.

\begin{theorem}[The feasible interpolation theorem, \cite{krajicek2019proof}]\label{INTERPOL}
Assume that the set of clauses $\{A_1,...,A_m,B_1,...,B_l\}$
for all $i\leq m, j \leq l$ satisfies 
\begin{align*}
&A_i \subseteq \{p_1,\neg p_1,...,p_n,\neg p_n,q_1,\neg q_1,...,q_s,\neg q_s\};  \\
& B_j \subseteq \{p_1,\neg p_1,...,p_n,\neg p_n,r_1,\neg r_1,...,r_t,\neg r_t\},
\end{align*}
and has a resolution refutation with $k$ clauses. Then the implication  
\begin{align*}
&\bigwedge_{i\leq m}(\bigvee A_i) \to \neg \bigwedge_{j \leq l}(\bigvee B_j)
\end{align*}
 has an interpolating circuit $I(\bar{p})$ whose size is $O(kn)$. If the refutation is tree-like, $I$ is a formula. Moreover, if all atoms $\bar{p}$ occur only positively in all $A_i$, then there is a monotone interpolating circuit (or a formula in the tree-like case) whose size is $O(kn)$.
\end{theorem}

\section{Lower Bounds}
In this section we consider another side of the Dichotomy of the $\mathcal{H}$-coloring problem, namely, $NP$-complete case for non-bipartite graphs $\mathcal{H}$. Well-studied example of the $\mathcal{H}$-coloring problem is the $\mathcal{K}_n$-coloring problem, which is essentially the $n$-coloring problem, where $\mathcal{K}_n$ is a complete graph on $n>2$ vertices. One of the obvious negative instances for CSP($\mathcal{K}_n$) is the graph $\mathcal{K}_{n+1}$: it is impossible to $n$-color complete graph with $n+1$ vertices. Propositional formula, expressing that there is no homomorphism from $\mathcal{K}_{n+1}$ to $\mathcal{K}_n$, can be reduced to the Pigeonhole Principle formula  PHP$^{n+1}_n$, because essentially trying to find a homomorphism from $\mathcal{K}_{n+1}$ to $\mathcal{K}_n$ is trying to map injectively the set $[0,n+1]$ to the set $[0,n]$. The PHP$^{n+1}_n$ formula is:

\begin{equation}
\neg[\underset{i}\bigwedge\underset{j}\bigvee p_{ij}\wedge \underset{i}\bigwedge\underset{j \neq j' }\bigwedge (\neg p_{ij} \vee \neg p_{ij'}) \wedge \underset{i \neq i'}\bigwedge\underset{j}\bigwedge (\neg p_{ij} \vee \neg p_{i'j})],
\end{equation}
where $(n+1)n$ atoms $p_{ij}$ with $i\in [n+1]$ and $j \in [n]$ expressing that $i$ is mapped to $j$. For PHP$^{n+1}_n$ there is a lot of known lower bounds in different weak proof systems:

\begin{theorem}[\cite{HAKEN1985297}] There exists a constant $c$, $c>1$, so that, for suffisiently large $n$, every resolution refutation of $\neg$PHP$^{n+1}_n$ contains al teast $c^n$ different clauses.
\end{theorem}

\begin{theorem}[Ajtai 1988, Beame et al. 1992, \cite{krajicek_1995}] Assume that $F$ is a Frege proof system and $d$ is a constant, and let $n>1$.
Then in every depth $d$ $F$-proof of the formula PHP$^{n+1}_n$ at least $2^{n^{(1/6)^d}}$ different formulas must occur. In particular, each depth $d$ $F$-proof of PHP$^{n+1}_n$ must have size at least $2^{n^{(1/6)^d}}$ and must have at least $\Omega(2^{n^{(1/6)^d}})$ proof steps.
\end{theorem}
We also can consider weak variants of PHP principle, PHP$_n^m$, where the number $m$ of pigeons is larger then $n+1$ (which will be equivalent to non-existence of homomorphism from $\mathcal{K}_m$ to $\mathcal{K}_n$). 
\begin{theorem}[\cite{RazborovLoverBounds}] For $m>n$ PHP$^m_n$ has no polynomial calculus refutation of degree $d\leq \lceil n/2\rceil$.
\end{theorem}

\begin{theorem}[\cite{10.1007/BFb0029951}] Let $c,d$ and a prime $p$ be fixed, and let $q$ be a number not divisible by $p$. Then there is $\delta>0$ such that for all $n$ large enough it holds: there is $m\leq n$ such that in every tree-like $F_d^c(MOD_p)$-proof of PHP$_n^{n+m}$ at least $exp(n^{\delta})$ different formulas must occur. 
 \end{theorem}

Thus, we see that even for such an elementary negative instance of $NP$-complete case of the $\mathcal{H}$-coloring problem, CSP($\mathcal{K}_n$), the tautology, expressing that there is no homomorphism from $\mathcal{K}_{m}$ to $\mathcal{K}_n$, $m\geq n+1$, has no short proofs in many weak proof systems.  

\section{Conclusion}
We have constructed in Sec. 3.3 short proofs of propositional statements
expressing that $\mathcal{G} \notin CSP(\mathcal{H})$ for non-bipartite graphs $\mathcal{G}$ and bipartite graphs $\mathcal{H}$ by translating into propositional logic a suitable formalization of the algorithm for the $p$-time case of the $\mathcal{H}$-coloring problem. Note that while this algorithm is very simple, it is not $AC^0$-computable (parity is easily $AC^0$-reducible to the question whether or not a graph is bipartite) while our propositional proofs operate only with clauses and are thus, in this respect, more rudimentary than the decision algorithm is.

The condition for the $p$-time case of the $\mathcal{H}$-coloring problem (and the algorithm) are so simple that one could perhaps directly construct short propositional proofs and the use of bounded arithmetic may seem redundant. However, we think of this work as a stepping block towards proving analogous result for the full Dichotomy theorem. Its known proofs rely on universal algebra and formalizing them in a suitable bounded
arithmetic theory ought to be accessible while direct propositional
formalization looks unlikely. For this reason we used
bounded arithmetic here as a common framework. Moreover, this framework generally allows to obtain some collateral results that help to compose a complete picture of the problem.

An interesting issue which we left out is to prove a lower
bound not just for suitable $\mathcal{H}$ (as we did in Sec.4) but for
all $\mathcal{H}$ which fall under the $NP$-complete case of the Dichotomy
theorem. If CSP($\mathcal{H}$) is $NP$-complete then, unless $NP = coNP$, no
proof system can prove in $p$-size all valid statements
$\mathcal{G} \notin$ CSP($\mathcal{H}$). In addition, if the $NP$-completeness of the class can be formalized in a theory $T$ and we have a lower bound for
the proof system corresponding to $T$ (see \cite{krajicek2019proof} for this
topic) then one can use it to construct $\mathcal{G}$ for which the lower bound
holds. This uses well-known part of proof complexity but
we do feel that it adds to our understanding of the proof complexity
of CSP; it is rather a transposition of known results via
known techniques. For this reason we do not pursue here this
avenue of research.\color{black}

\bigskip
\noindent
\textbf{Acknowledgements:}{ I would like to thank my supervisor Jan Krajíček for helpful comments that resulted in many improvements to this paper. Also, I'm grateful to Pavel Pudlák, Neil Thapen, and others for the opportunity to present this work in the Institute of Mathematics of the Czech Academy of Sciences and for further discussion. Finally, I would like to thank Albert Atserias and Joanna Ochremiak, whose paper \cite{atserias2017proof} inspired this direction of research.}

\bibliographystyle{plain}    

\bibliography{bibliography}

\end{document}